\documentclass[numbib]{imamat}

\input{epsf.sty}
\usepackage{subfigure}
\usepackage{graphicx}
\usepackage{indentfirst}
\usepackage{fancyhdr}
\usepackage{amsmath}
\usepackage{amsthm}
\usepackage{amsfonts}
\usepackage{amssymb}
\usepackage{mathrsfs}

\usepackage{color}

\newcommand{\gra}{\textcolor[rgb]{0.85,0.85,0.85}}

\begin{document}

\title{
Analysis of random walks on a hexagonal lattice
}

\shorttitle{Analysis of random walks on a hexagonal lattice} 
\shortauthorlist{A. Di Crescenzo, C. Macci, B. Martinucci, S. Spina} 

\author{
\name{Antonio Di Crescenzo$^*$}
\address{Dipartimento di Matematica, 
Universit\`a degli Studi di Salerno, Via Giovanni Paolo II n.\ 132, 
84084 Fisciano, SA, Italy\email{$^*$Corresponding author: adicrescenzo@unisa.it}}
\name{Claudio Macci}
\address{Dipartimento di Matematica, 
Universit\`a di Roma Tor Vergata, Via della Ricerca Scientifica, 
00133 Rome, Italy}
\name{Barbara Martinucci}
\address{Dipartimento di Matematica, 
Universit\`a degli Studi di Salerno, Via Giovanni Paolo II n.\ 132, 
84084 Fisciano, SA, Italy}
\and
\name{Serena Spina}
\address{Dipartimento di Matematica, 
Universit\`a degli Studi di Salerno, Via Giovanni Paolo II n.\ 132,
84084 Fisciano, SA, Italy}
}

\maketitle

\begin{abstract}
{
We consider a discrete-time random walk on the nodes of an unbounded hexagonal lattice. 
We determine the probability generating functions,  
the transition probabilities and the relevant moments. The convergence of the stochastic 
process to a 2-dimensional Brownian motion is also discussed. Furthermore, we obtain some 
results on its asymptotic behavior making use of large deviation theory. 
Finally, we investigate the first-passage-time problem of the random walk through a vertical 
straight-line. Under suitable symmetry assumptions we are able to determine the first-passage-time 
probabilities in a closed form, which deserve interest in applied fields. 
}
{Random walk; Hexagonal lattice; 
Probability generating function;  Large deviations; Moderate deviations; First-passage time.}
\\
2000 Math Subject Classification: 60J15; 60F10; 82C41  
\end{abstract}

\section{Introduction}
Stimulated by potential applications in many fields of science and engineering, in this paper 
we aim to study a discrete-time random walk on the nodes of an unbounded hexagonal lattice. 
Specific properties of this kind of structures make them attractive for various applications, such as 
thermal isolation, energy absorption, and structural protection. For a more detailed description on the use 
of honeycomb structures in applied fields see \cite{Haghpanah}. 
Some general results on discrete-time random walks on a lattice 
can be found in \cite{Montroll}, \cite{MontrollWeiss} and in \cite{LawlerLimic}. 
See also the investigation of \cite{Guillotin} concerning random walks on regular graphs, 
and the recent review by \cite{Masuda_etal}.
\par
Our attention focuses on certain mathematical properties of the stochastic process 
under investigation, where 
the underlying lattice state-space is a general honeycomb structure (the hexagonal lattice). 
Specifically, we give emphasis on the transient distribution of the random walk. We also 
study its asymptotic behavior making use of the theory of large deviations. 
In view of the relevance in several applied contexts, 
our efforts are finally oriented to the determination of the first-passage-time 
probabilities of the random walk through suitable straight-line boundaries. 
\par
Two-dimensional random walk models on hexagonal structures deserve interest in 
various applied fields, such as 
Biomathematics, Cellular networks, Physics, and Chemical models.
\par
A correlated random walk on a hexagonal lattice has been 
used by \cite{Prasad} in order to determine the optimal movement 
strategy of an animal searching for resources upon a network of patches.
Moreover, hexagonal lattices have been adopted to characterize landscapes in suitable 
spatial models of bird populations, since hexagons allow better packing of territories in space, 
see \cite{Pulliam}. 
In Personal Communications Services networks, such as the honeycomb Poisson-Voronoi access cellular network 
model, the movement of mobile users may be  captured by random walks models on the hexagonal lattice
(cf.\  \cite{Akyildiz}, \cite{Baccelli}). 
We remark that the above mentioned investigations deal with random walks among adjacent cells. 
\par
Other types of processes dealing with two-dimensional random walks on hexagonal structures
are employed to study:  
the behavior of cracks among frozen regions in a dimer model 
(cf.\   \cite{Boutillier}), the representation of the correlation functions in valence-bond 
solid models (cf.\  \cite{Kennedy}), the light transport in a honeycomb 
structure (cf.\   \cite{Miri}),  electronic properties of deformed carbon 
nanotubes (cf.\   \cite{Schuyler}), and the incoherent energy 
transfer due to long range interactions in two-dimensional regular systems 
(cf.\   \cite{Zumofen}). 
\par   
A mathematical model based on a three-axes description
of the honeycomb lattice has been proposed by    \cite{Cotfas}, where the 
movement of an excitation (or a vacancy) on a quasicrystal is regarded as a suitable 
random walk. A random-walk  model of absorption of an isolated polymer chain 
on various lattice models, including the hexagonal one, is investigated in    \cite{Rubin}. 
Moreover, as a model of polymer dynamics,  \cite{Sokolov} analyzed the 
continuous-time motion of a rigid equilateral triangle in the plane, where the center of mass 
of the triangle performs a random walk on the vertices of the hexagonal lattice. 
%
\par  
Our study is first finalized to obtain a closed-form result of the probability generating function and 
of the probability distribution of the random walk on the hexagonal lattice. This is performed by 
considering a partition of the state space into two sets, i.e.\ the states visited at even and odd times. 
Then, the iterative equations of the relevant probabilities for the two sets are expressed in a  
suitable way. A similar approach has been used by \cite{DMM2014} for the analysis of random walks in continuous 
time characterized by alternating rates. Some auxiliary results are also obtained, such as 
certain symmetry properties of the state probabilities and the relevant moments, 
including the covariance. The validity 
of a customary convergence to a 2-dimensional Brownian motion is also shown. 
Differently from previous investigations oriented to computing numerical 
quantities of interest, such as critical exponents (see  \cite{de_Forcrand}, 
for instance), our approach is mainly theoretical. Indeed, we remark that our 
study leads to closed-form results even in the general case of non-constant 
one-step transition probabilities. For general results on other types of discrete-time 
random walks see the contributions by \cite{Katzenbeisser},  \cite{Bohm} 
and \cite{Panny}. 
\par
Our second aim is to investigate two different forms of asymptotic behavior of the 
random walk on the hexagonal lattice, by making use of some applications of the 
G\"artner Ellis Theorem. About this topic we recall the text by   \cite{FK2006} 
for a wide study on sample path large deviations for general Markov processes. 
It is worth pointing out that our  results concerning the large deviation principle (LDP for short) 
can be used to evaluate some quantities of interest in applied contexts, such as estimates 
of hitting probabilities that are relevant for Monte Carlo simulation through the importance 
sampling technique (see, for instance, the overview in \cite{BlanchetLam2012}, or 
Section 5 and 6 of \cite{Asmussen}, and \cite{Collamore2002}.) 
\par
In conclusion, we also face the first-passage-time problem of the considered random walk 
through straight-line boundaries. We first investigate the ``taboo probability'' concerning  
transitions on the nodes of the hexagonal lattice that avoid the boundary. Suitable 
symmetry conditions allow us to obtain closed-form expressions of such taboo probabilities and, 
in turn, of the corresponding first-passage-time probabilities. 
An example of the usefulness of these results  in the context of polymer science 
is  provided at the end of  Section 5. We note that obtaining closed-form expressions for 
random walks on lattices generally is not an easy task. 
\par 
Here is the plan of the paper:  The main definitions and the description of the process 
are given in Section 2. The probability distribution of the random walk, with the 
generating function and the main moments are obtained in Section 3. 
Section 4 is devoted to investigation of the large and moderate deviations. 
In Section 5 we analyze the first-passage-time problem of the random walk through 
straight-line boundaries. Finally, in Section 6 we provide some concluding remarks.
\section{The Random Walk Model}
We consider the hexagonal lattice on a reference system of cartesian axes, taking  a
vertex of a generic hexagon as the origin of the reference system, as shown in Fig.\ \ref{fig:stati}.
Since the considered structure consists of hexagonal cells, with angles of $2\pi/3$, we can assume 
that the distance between two generic adjacent vertices is a constant, say $a$. 
Hence, the coordinates of the vertices are repeated regularly.
We divide the vertices into two categories:
\begin{equation}\label{def_nodes}
\mathscr{V}_i=\left\{\left(\frac{3}{2}a j +ia;\; \frac{\sqrt{3}}{2} a j +\sqrt{3}a k\right);\;j,k \in \mathbb{Z}\right\},\quad i=0,1.
\end{equation}
\begin{figure}[ht]
\begin{center}
\begin{picture}(341,221)
\put(50,15){\circle{5}}
\put(130,15){\circle{5}}
\put(210,15){\circle{5}}
\put(75,15){\circle*{5}}
\put(155,15){\circle*{5}}
\put(235,15){\circle*{5}}
\put(50,60){\circle{5}}
\put(130,60){\circle{5}}
\put(210,60){\circle{5}}
\put(290,60){\circle{5}}
\put(75,60){\circle*{5}}
\put(155,60){\circle*{5}}
\put(235,60){\circle*{5}}
\put(50,105){\circle{5}}
\put(130,105){\circle{5}}
\put(210,105){\circle{5}}
\put(290,105){\circle{5}}
\put(75,105){\circle*{5}}
\put(155,105){\circle*{5}}
\put(235,105){\circle*{5}}
\put(50,150){\circle{5}}
\put(130,150){\circle{5}}
\put(210,150){\circle{5}}
\put(290,150){\circle{5}}
\put(75,150){\circle*{5}}
\put(155,150){\circle*{5}}
\put(235,150){\circle*{5}}
\put(90,37.5){\circle{5}}
\put(170,37.5){\circle{5}}
\put(250,37.5){\circle{5}}
\put(35,37.5){\circle*{5}}
\put(115,37.5){\circle*{5}}
\put(195,37.5){\circle*{5}}
\put(275,37.5){\circle*{5}}
\put(90,82.5){\circle{5}}
\put(170,82.5){\circle{5}}
\put(250,82.5){\circle{5}}
\put(35,82.5){\circle*{5}}
\put(115,82.5){\circle*{5}}
\put(195,82.5){\circle*{5}}
\put(275,82.5){\circle*{5}}
\put(90,127.5){\circle{5}}
\put(170,127.5){\circle{5}}
\put(250,127.5){\circle{5}}
\put(35,127.5){\circle*{5}}
\put(115,127.5){\circle*{5}}
\put(195,127.5){\circle*{5}}
\put(275,127.5){\circle*{5}}
\put(90,172.5){\circle{5}}
\put(170,172.5){\circle{5}}
\put(250,172.5){\circle{5}}
\put(115,172.5){\circle*{5}}
\put(195,172.5){\circle*{5}}
\put(275,172.5){\circle*{5}}
\put(120,82.5){\line(1,0){5}}
\put(130,82.5){\line(1,0){5}}
\put(140,82.5){\line(1,0){5}}
\put(150,82.5){\line(1,0){5}}
\put(160,82.5){\line(1,0){5}}
\put(40,82.5){\line(1,0){5}}
\put(50,82.5){\line(1,0){5}}
\put(60,82.5){\line(1,0){5}}
\put(70,82.5){\line(1,0){5}}
\put(80,82.5){\line(1,0){5}}
\put(25,82.5){\line(1,0){5}}
\put(15,82.5){\line(1,0){5}}
\put(200,82.5){\line(1,0){5}}
\put(210,82.5){\line(1,0){5}}
\put(220,82.5){\line(1,0){5}}
\put(230,82.5){\line(1,0){5}}
\put(240,82.5){\line(1,0){5}}
\put(280,82.5){\line(1,0){5}}
\put(290,82.5){\line(1,0){5}}
\put(300,82.5){\line(1,0){5}}
\put(310,82.5){\vector(1,0){10}}
\put(170,-2.5){\line(0,1){5}}
\put(170,7.5){\line(0,1){5}}
\put(170,17.5){\line(0,1){5}}
\put(170,27.5){\line(0,1){5}}
\put(170,42.5){\line(0,1){5}}
\put(170,52.5){\line(0,1){5}}
\put(170,62.5){\line(0,1){5}}
\put(170,72.5){\line(0,1){5}}
\put(170,87.5){\line(0,1){5}}
\put(170,97.5){\line(0,1){5}}
\put(170,107.5){\line(0,1){5}}
\put(170,117.5){\line(0,1){5}}
\put(170,132.5){\line(0,1){5}}
\put(170,142.5){\line(0,1){5}}
\put(170,152.5){\line(0,1){5}}
\put(170,162.5){\line(0,1){5}}
\put(170,177.5){\line(0,1){5}}
\put(170,187.5){\line(0,1){5}}
\put(170,197.5){\vector(0,1){10}}
\put(50,60){\line(1,0){25}}
\put(50,60){\line(-2,3){15}}
\put(50,60){\line(-2,-3){15}}
\put(130,60){\line(1,0){25}}
\put(130,60){\line(-2,3){15}}
\put(130,60){\line(-2,-3){15}}
\put(210,60){\line(1,0){25}}
\put(210,60){\line(-2,3){15}}
\put(210,60){\line(-2,-3){15}}
\put(50, 105){\line(1,0){25}}
\put(50, 105){\line(-2,3){15}}
\put(50, 105){\line(-2,-3){15}}
\put(130, 105){\line(1,0){25}}
\put(130, 105){\line(-2,3){15}}
\put(130, 105){\line(-2,-3){15}}
\put(210, 105){\line(1,0){25}}
\put(210, 105){\line(-2,3){15}}
\put(210, 105){\line(-2,-3){15}}
\put(90, 82.5){\line(1,0){25}}
\put(90, 82.5){\line(-2,3){15}}
\put(90, 82.5){\line(-2,-3){15}}
\put(170, 82.5){\line(1,0){25}}
\put(170, 82.5){\line(-2,3){15}}
\put(170, 82.5){\line(-2,-3){15}}
\put(250, 82.5){\line(1,0){25}}
\put(250, 82.5){\line(-2,3){15}}
\put(250, 82.5){\line(-2,-3){15}}
\put(130, 150){\line(1,0){25}}
\put(130, 150){\line(-2,3){15}}
\put(130, 150){\line(-2,-3){15}}
\put(210, 150){\line(1,0){25}}
\put(210, 150){\line(-2,3){15}}
\put(210, 150){\line(-2,-3){15}}
\put(290, 150){\line(-2,3){15}}
\put(290, 150){\line(-2,-3){15}}
\put(90, 127.5){\line(1,0){25}}
\put(90, 127.5){\line(-2,3){15}}
\put(90, 127.5){\line(-2,-3){15}}
\put(170, 127.5){\line(1,0){25}}
\put(170, 127.5){\line(-2,3){15}}
\put(170, 127.5){\line(-2,-3){15}}
\put(250, 127.5){\line(1,0){25}}
\put(250, 127.5){\line(-2,3){15}}
\put(250, 127.5){\line(-2,-3){15}}
\put(35, 127.5){\line(2,3){15}}
\put(50, 150){\line(1,0){25}}
\put(75, 150){\line(2,3){15}}
\put(90, 172.5){\line(1,0){25}}
\put(155, 150){\line(2,3){15}}
\put(170, 172.5){\line(1,0){25}}
\put(235, 150){\line(2,3){15}}
\put(250, 172.5){\line(1,0){25}}
\put(275, 127.5){\line(2,-3){15}}
\put(275, 82.5){\line(2,3){15}}
\put(35, 37.5){\line(2,-3){15}}
\put(50, 15){\line(1,0){25}}
\put(75, 15){\line(2,3){15}}
\put(75, 15){\line(2,3){15}}
\put(90, 37.5){\line(-2,3){15}}
\put(90, 37.5){\line(1,0){25}}
\put(115, 37.5){\line(2,-3){15}}
\put(130, 15){\line(1,0){25}}
\put(155, 15){\line(2,3){15}}
\put(170, 37.5){\line(-2,3){15}}
\put(170, 37.5){\line(1,0){25}}
\put(195, 37.5){\line(2,-3){15}}
\put(210, 15){\line(1,0){25}}
\put(235, 15){\line(2,3){15}}
\put(250, 37.5){\line(-2,3){15}}
\put(250, 37.5){\line(1,0){25}}
\put(275, 37.5){\line(2,3){15}}
\put(275, 37.5){\line(2,3){15}}
\put(290, 60){\line(-2,3){15}}
\put(210, -1){\line(0,1){4}}
\put(235, -1){\line(0,1){4}}
\put(210, 1){\line(1,0){25}}
\put(198,-19){\makebox(50,15)[t]{$a$}}
\put(300,80){\makebox(50,15)[t]{$x$}}
\put(155,200){\makebox(50,15)[t]{$y$}}
\end{picture}
\end{center}
\caption{Graphical representation of the hexagonal lattice, where the vertices of $\mathscr{V}_0$
$(\mathscr{V}_1)$ are represented by white (black) circles.}
\label{fig:stati}
\end{figure}
\begin{figure}[ht]
\begin{center}
\begin{picture}(341,121)
\put(60,60){\vector(1,0){45}}
\put(60,60){\vector(-2,3){25}}
\put(60,60){\vector(-2,-3){25}}
\put(220,60){\vector(-1,0){45}}
\put(220,60){\vector(2,3){25}}
\put(220,60){\vector(2,-3){25}}
\put(80,60){\makebox(50,15)[t]{$q_{0,0}$}}
\put(10,95){\makebox(50,15)[t]{$q_{0,1}$}}
\put(10,0){\makebox(50,15)[t]{$q_{0,2}$}}
\put(150,60){\makebox(50,15)[t]{$q_{1,0}$}}
\put(230,95){\makebox(50,15)[t]{$q_{1,2}$}}
\put(230,0){\makebox(50,15)[t]{$q_{1,1}$}}
\put(60,60){\circle{5}}
\put(220,60){\circle*{5}}
\end{picture}
\end{center}
\caption{One-step transition probabilities.}
\label{fig:transprob}
\end{figure}
\par 
With reference to Fig.\ \ref{fig:stati},
we consider a random walk of a particle that starts at a vertex with coordinates $(x,y)$ and it moves to an
adjacent vertex following an appropriate transition probability. In particular, as shown in Fig.\ \ref{fig:transprob},
if the particle is located in a vertex of the set $\mathscr{V}_0$, it can reach the three adjacent positions
with probabilities $q_{0,0}$, $q_{0,1}$, $q_{0,2}$ and then the particle will occupy a vertex
of $\mathscr{V}_1$. Similarly, if the particle is in a vertex of $\mathscr{V}_1$, in one step it reaches
one of the three adjacent positions, belonging to $\mathscr{V}_0$, with probabilities
$q_{1,0}$, $q_{1,1}$, $q_{1,2}$ (see Fig.\ \ref{fig:transprob}). 
We remark that the hexagonal graph is a bipartite graph, this property being useful for the evaluation 
of the results of Section 3.
\par
Let  $\{(X_n,Y_n), n\in\mathbb{N}_0\}$ be the discrete-time random walk having state space
$\mathscr{V}_0 \cup \mathscr{V}_1$ and representing the position of the particle at time $n$.
From the above 
notations, for $i=0,1$ and $r=0,1,2$ the one-step transition probabilities 
are expressed as 
\begin{equation}\label{eq:qir}
q_{i,r}=\mathbb{P}\left[\binom{X_{n+1}}{Y_{n+1}}
=\binom{x+a \cos\left(r\frac{2}{3}\pi+i \pi\right)}{y+a \sin\left(r\frac{2}{3}\pi + i \pi\right)}\left|
 \binom{X_n}{Y_n}=\binom{x}{y}\right.\right],
\end{equation}
with
$$
 \sum_{r=0}^2 q_{i,r}=1,\qquad i=0,1.
$$
Let us now introduce the state probabilities at time $n$, $n\in\mathbb{N}_0$, 
\begin{equation}\label{p_j_k}
p_{j,k}(n):=\mathbb{P}\left[(X_{n},Y_{n})=\left(\frac{3}{2} a j+i_n a,
\frac{\sqrt{3}}{2} a j +\sqrt{3}a k\right)\right],
\quad j,k\in \mathbb{Z},
\end{equation}
where
\begin{equation}\label{eq:i_n}
i_n=\frac{1}{2}\left(1-(-1)^n\right).
\end{equation}
We assume that $\mathbb{P}[(X_0, Y_0)=(0,0)]=1$, so that the initial condition reads 
\begin{equation}\label{eq:iniziale}
 p_{0,0}(0)=1.
\end{equation}
Hence, from the previous assumptions, and noting that at even (odd) times the particle
occupies states of set $\mathscr{V}_0$  ($\mathscr{V}_1$), the forward Kolmogorov
equations for the state probabilities, for all $n\in \mathbb{N}_0$ are given by
\begin{equation}\label{eq:transprob}
 p_{j,k}(n+1)=
 \left\{
 \begin{array}{ll}
 p_{j,k}(n)q_{0,0}+p_{j+1,k}(n)q_{0,2}+p_{j+1,k-1}(n) q_{0,1}, & n \hbox{ even} \\[2mm]
 p_{j,k}(n)q_{1,0}+p_{j-1,k+1}(n)q_{1,1}+p_{j-1,k-1}(n) q_{1,2}, & n \hbox{ odd}
\end{array}
\right.
\end{equation}
with initial condition (\ref{eq:iniziale}).

%

\section{Probability Distribution}
Now we focus on the probability generating function $G(u,v;n)$ of
$(X_n,Y_n)$. Due to (\ref{p_j_k}), it is defined by 
\begin{equation}\label{prob_gen_func_def}
G(u,v;n)=\mathbb{E}\left[u^{X_n}v^{Y_n}\right]
=\sum_{j\in\mathbb{Z}}u^{\frac{3}{2}aj+i_na}\sum_{k\in\mathbb{Z}}v^{\frac{\sqrt{3}}{2}aj+\sqrt{3}ak}p_{j,k}(n),
\end{equation}
for $u>0$ and $v>0$. 
\begin{proposition}\label{prop:1}
The explicit expression of the probability generating function $G(u,v;n)$ of $(X_n,Y_n)$, 
for $n\in \mathbb{N}_0$, $u>0$ and $v>0$ is:
\begin{equation}\label{prob_gen_func}
G(u,v;n)=F_{i_n}(\tilde{u},\tilde{v};n) u^{i_n a},
\end{equation}
where $i_n$ is defined by (\ref{eq:i_n}), with
\begin{equation}\label{u_v_tilde}
\tilde{u}=u^{\frac{3}{2}a}v^{\frac{\sqrt{3}}{2}a},\qquad\tilde{v}=v^{\sqrt{3}a},
\end{equation}
and
\begin{equation}\label{F_i_expl}
F_{i_n}(u,v;n)=\left( q_{1,0}+q_{1,2}u+q_{1,1}\frac{u}{v} \right)^{\frac{n}{2}-\frac{1}{2}{i_n}}
\left( q_{0,0}+q_{0,1}\frac{v}{u}+q_{0,2}\frac{1}{u}\right)^{\frac{n}{2}+\frac{1}{2}{i_n}}.
\end{equation}
%
\end{proposition}
\begin{proof}
For all $n\in \mathbb{N}_0$, let us define 
\begin{equation}\label{F_i_def}
F_{i_n}\left(u,v;n \right)= \sum_{j=-\infty}^{+\infty}u^j \sum_{k=-\infty}^{+\infty}v^k p_{j,k}(n), 
\end{equation}
so that for $n$ even (odd),  $F_{0}(u,v;n)$ ($F_{1}(u,v;n)$) is the probability generating function 
of $(X_n,Y_n)$ for the vertex set $\mathscr{V}_0$ ($\mathscr{V}_1$) defined by  
(\ref{def_nodes}). Due to Eq.\ (\ref{eq:transprob}) we have
\begin{equation}\label{system}
\phi^{(n+1)}=M \phi^{(n)},
\end{equation}
where
\begin{equation}\label{def_M_phi}
M=\left(
\begin{array}{ll}
0 & \alpha  \\
\beta  &  0
\end{array}
\right), \qquad \phi^{(n)}={F_0(u,v;n) \choose F_1(u,v;n)},
\end{equation}
and
\begin{equation}\label{eq:alfabeta}
\begin{array}{l}
\alpha=\alpha(u,v)=q_{1,0}+q_{1,2}u+q_{1,1}\frac{u}{v}, \\[2mm]
\beta=\beta(u,v)=q_{0,0}+q_{0,1}\frac{v}{u}+q_{0,2}\frac{1}{u}.
\end{array}
\end{equation}
Due to  initial condition (\ref{eq:iniziale}) we have $\phi^{(0)}={1 \choose 0}$,
so that the system (\ref{system}) has solution:
$$\phi^{(n)}=M^n \phi^{(0)},$$
where
$$M^n=\left(
\begin{array}{ll}
\alpha^{\frac{n}{2}}\beta^{\frac{n}{2}} & \quad 0\\
0 & \quad \alpha^{\frac{n}{2}}\beta^{\frac{n}{2}}
\end{array}
\right),\qquad \textrm{ if $n$ is even},
$$
$$M^n=\left(
\begin{array}{ll}
0 & \alpha^{\frac{n+1}{2}}\beta^{\frac{n-1}{2}} \\
\alpha^{\frac{n+1}{2}}\beta^{\frac{n-1}{2}} &  0
\end{array}
\right),\qquad \textrm{ if $n$ is odd}.
$$
Recalling (\ref{def_M_phi}), we have the following explicit expressions:
\begin{equation}\label{F_01_expl}
F_0\left(u,v;n \right)=\left\{
\begin{array}{ll}
\alpha^{\frac{n}{2}}\beta^{\frac{n}{2}},& \textrm{$n$  even}\\
0, & \textrm{$n$ odd},
\end{array}
\right.
\quad
F_1\left(u,v;n \right)=\left\{
\begin{array}{ll}
0,& \textrm{$n$ even}\\
\alpha^{\frac{n-1}{2}}\beta^{\frac{n+1}{2}}, & \textrm{$n$ odd}.
\end{array}
\right.
\end{equation}
Now, noting that the generating function (\ref{prob_gen_func_def}) can be written,
due to (\ref{F_i_def}), as
$$
 G(u,v;n)=\left\{
\begin{array}{ll}
F_0(\tilde{u},\tilde{v};n),& n\textrm{ even}\\
&\\
F_1(\tilde{u},\tilde{v};n) u^a,& n\textrm{ odd},
\end{array}\right.
$$
by comparing this last expression with (\ref{F_i_def}),  (\ref{eq:alfabeta}), (\ref{F_01_expl}),
and by taking into account (\ref{u_v_tilde}), Eq.\ (\ref{prob_gen_func}) thus follows.
\end{proof}
\par
Let us recall the Gauss hypergeometric function
\begin{equation}
 {}_{2}F_{1}(a,b;c;z)=\sum_{n=0}^{+\infty}
 \frac{(a)_n (b)_n}{(c)_n}\,\frac{z^n}{n!},
\label{2F1}
\end{equation}
where  $(a)_n$ is the Pochhammer symbol 
defined by $(a)_n=a(a+1)\ldots(a+n-1)$ for $n\in \mathbb{N}$ and  $(a)_0=1$.
\par
We recall that, due to initial condition (\ref{eq:iniziale}), the random walk $(X_n,Y_n)$ occupies the 
states of the vertex set ${\cal V}_0$ $({\cal V}_1)$ at even (odd) times, see Eq.\ (\ref{def_nodes}). 
We now determine the state probabilities of $(X_n,Y_n)$ at even times.
\begin{proposition}\label{propProb}
Let $m\in \mathbb{N}$, and 
\begin{equation}
 \rho= \frac{q_{0,1} q_{1,1}}{q_{0,2} q_{1,2}}.
 \label{eq:defrho}
\end{equation}
(i) For $0\leq j\leq m$ and $-m\leq k\leq 0$,
\begin{eqnarray*}
p_{j,k}(2 m) &=& \sum_{t=0}^{m-j} {m \choose t}{m \choose j+t} {j+t \choose -k}
q_{0,0}^{m-t}\, q_{1,0}^{m-j-t}\, q_{0,2}^{t}\, q_{1,2}^{j+k+t}\, q_{1,1}^{-k}   \\
&\times &  {}_{2}F_{1}\left(-j-k-t, -t; 1-k;\rho\right).
\end{eqnarray*}
(ii) For $0\leq j\leq m$ and $1\leq k\leq m-j$,
\begin{eqnarray*}
p_{j,k}(2 m) &=& \sum_{t=k}^{m-j} {m \choose t}{m \choose j+t} {t \choose k}
q_{0,0}^{m-t}\, q_{1,0}^{m-j-t}\, q_{0,1}^{k}\, q_{0,2}^{-k+t}\, q_{1,2}^{j+t} \\
&\times &{}_{2}F_{1}\left(-j-t, k-t; 1+k;\rho\right).
\end{eqnarray*}
(iii) For $-m\leq j\leq -1$ and $-m-j\leq k\leq -1$,
\begin{eqnarray*}
p_{j,k}(2 m) &=& \sum_{t=-k}^{m+j} {m \choose t}{m \choose -j+t} {t \choose -k}
q_{0,0}^{m+j-t}\, q_{1,0}^{m-t}\, q_{0,2}^{-j+t}\, q_{1,2}^{k+t} \, q_{1,1}^{-k} \\
&\times & {}_{2}F_{1}\left(j-t, -k-t; 1-k;\rho\right).
\end{eqnarray*}
(iv) For $-m\leq j\leq -1$ and $0\leq k\leq m$,
\begin{eqnarray*}
p_{j,k}(2 m) &=& \sum_{t=0}^{m+j} {m \choose t}{m \choose -j+t} {-j+t \choose k}
q_{0,0}^{m-t}\, q_{1,0}^{m+j-t}\, q_{0,1}^{k}\, q_{0,2}^{-j-k+t}\, q_{1,2}^{t}  \\
&\times & {}_{2}F_{1}\left(j+k-t, -t; 1+k;\rho\right).
\end{eqnarray*}
\end{proposition}
\begin{proof}
By extracting the coefficients of $u^j$ and $v^k$ in Eq.\ (\ref{F_i_def}) for $i=0$,
making use of definition (\ref{2F1}) and recalling that
$$
(x)_n=\frac{(-1)^n}{(1-x)_n},\quad n\in {\mathbb Z},
$$
the proof then follows after very cumbersome and tedious calculations, and then are omitted.
\end{proof}
%
%
%
\par
%
\par
Symmetry properties of two-dimensional stochastic processes are often encountered in various 
applications (see, for instance, Proposition 2.1 of  \cite{DM2008} for 
a family of two-dimensional continuous-time random walks). Hereafter we exploit various symmetry 
properties for the transition probabilities given in Proposition \ref{propProb}. 
\begin{corollary}\label{coroll:symm}
From Proposition \ref{propProb} the following symmetry properties hold for $m\in \mathbb{N}$. 
\\
(i) For $-m\leq j\leq m$ and  $-m\leq k\leq m$, if $q_{0,1}=q_{1,2}$ and $q_{1,1}=q_{0,2}$,  we have 
$$
 p_{j,k}(2 m)= p_{-j,j+k}(2 m).
$$
(ii) For $-m\leq j\leq m$ and  $-m\leq k\leq m$ we have 
$$
 p_{j,k}(2 m)=\xi^{2k+j} p_{j,-j-k}(2 m), \qquad \hbox{for \ } \frac{q_{0,1}}{q_{0,2}}=\frac{q_{1,2}}{q_{1,1}}=\xi.
$$
(iii) For $-m\leq j\leq m$ and  $-m\leq k\leq m$, if $q_{0,1}=q_{1,2}$ and $q_{1,1}=q_{0,2}$, we have 
$$
 p_{j,k}(2 m)=\delta^{2k+j} p_{-j,-k}(2 m), \qquad \hbox{for \ } \frac{q_{0,1}}{q_{1,1}}=\frac{q_{1,2}}{q_{0,2}}=\delta.
$$
\end{corollary}
\par
Specifically, with reference to Fig.\ 1, case (i) of Corollary \ref{coroll:symm} is concerning 
the symmetry with respect to $y$-axes, while case (ii) refers to $x$-axes, and case (iii) to the origin. 
\par
We remark that in all cases treated in Corollary \ref{coroll:symm}, 
from (\ref{eq:defrho}) we have $\rho=1$. 
This condition allows to simplify the expressions of the state probabilities 
given in Proposition \ref{propProb}, since 
${}_{2}F_{1}\left(a,b; c;1\right)=\frac{\Gamma(c)\Gamma(c-a-b)}{\Gamma(c-a)\Gamma(c-b)}$ 
(see, for instance, Eq.\ 15.1.20 of \cite{Abr}). 
\par
We point out that the state probabilities given in Proposition \ref{propProb} 
can be evaluated for odd times making use of equation (\ref{eq:transprob}).
Specifically, it is not hard to see that $p_{j,k}(2 m+1) >0$
for $0\leq j\leq m$ and $-m\leq k\leq m-j$, and
for $-m-1\leq j\leq -1$ and $-m-j-1\leq k\leq m+1$. Hence, Corollary \ref{coroll:symm} can be 
extended to the case of odd times. 
\par
Denoting by $\tilde{X}_i$ the displacement of the $i$-th step on the $x$ axis and by
$\tilde{Y}_i$ the displacement of the $i$-th step on the $y$ axis, for $i=1,2,\ldots,n$,
the random walk $(X_n,Y_n)$ can be expressed as
\begin{equation}\label{sum}
 S_n:=\binom{X_n}{ Y_n}
 =\binom{\tilde{X}_1+\tilde{X}_2+\ldots +\tilde{X}_n }{ \tilde{Y}_1+\tilde{Y}_2+\ldots +\tilde{Y}_n},
 \qquad n\in\mathbb{N},
\end{equation}
the joint distribution of $(\tilde{X}_i,\tilde{Y}_i)$ being given by Eq.\ (\ref{eq:qir}).
\par
Let us now determine some moments of interest, that will be used in the sequel.
\begin{proposition}\label{mean_var_cov}
For all $n \in \mathbb{N}$, for the random vector $S_n$ defined by (\ref{sum})
the vector mean is:
\begin{equation}
 m_n:=\left(
 \begin{array}{c}
 \mathbb{E}(X_n) \\[0.1cm]
 \mathbb{E}(Y_n)
 \end{array}
 \right)
 =
 \left(
 \begin{array}{c}
 \mu_1 n+ \theta_1 i_n \\[0.1cm]
 \mu_2 n+ \theta_2 i_n
 \end{array}
 \right),
 \label{eq:mean}
\end{equation}
where $i_n$ is defined by (\ref{eq:i_n}), and
$$
 \mu_1=\frac{3}{4}a\left[(q_{1,1}+q_{1,2})-(q_{0,1}+q_{0,2})\right],
 \quad
  \theta_1=a-\frac{3}{4} a \left[(q_{1,1}+q_{1,2})+(q_{0,1}+q_{0,2})\right],
$$
$$
 \mu_2=\frac{\sqrt{3}}{4}a\left[(q_{0,1}-q_{0,2})-(q_{1,1}-q_{1,2})\right],
 \quad
  \theta_2=\frac{\sqrt{3}}{4}a\left[(q_{0,1}-q_{0,2})+(q_{1,1}-q_{1,2})\right].
$$
The variance of $S_n$ is expressed as
$$
 \left(
 \begin{array}{c}
 Var(X_n) \\[0.1cm]
 Var(Y_n)
 \end{array}
 \right)
 =
 \left(
 \begin{array}{c}
 \sigma_1^2 n+ \theta_3 i_n \\[0.1cm]
 \sigma_2^2 n+ \theta_4 i_n
 \end{array}
 \right),
$$
where
$$
 \sigma_1^2=\frac{9}{8} a^2 \left[( q_{0,1}+q_{0,2})-( q_{0,1}+q_{0,2})^2+(q_{1,1}+q_{1,2})-(q_{1,1}+q_{1,2})^2\right],
$$
$$
  \theta_3=\frac{9}{8}a^2\left[( q_{0,1}+q_{0,2})-( q_{0,1}+q_{0,2})^2-(q_{1,1}+q_{1,2})+(q_{1,1}+q_{1,2})^2\right],
$$
$$
 \sigma_2^2=\frac{3}{8} a^2 \left[( q_{0,1}+q_{0,2})-( q_{0,1}-q_{0,2})^2+(q_{1,1}+q_{1,2})-(q_{1,1}-q_{1,2})^2\right],
$$
$$
  \theta_4=\frac{3}{8} a^2 \left[( q_{0,1}+q_{0,2})-( q_{0,1}-q_{0,2})^2-(q_{1,1}+q_{1,2})+(q_{1,1}-q_{1,2})^2\right].
$$
Finally, the expression of the covariance is
$$
Cov(X_n,Y_n)= \sigma_{1,2} n+ \theta_5 i_n,
$$
where
$$
 \sigma_{1,2}=3\frac{\sqrt{3}}{8}a^2 \left[q_{0,1}(q_{0,1}-1)+q_{0,2}(1-q_{0,2})-q_{1,1}(1-q_{1,1})+q_{1,2}(1-q_{1,2})\right],
$$
$$
  \theta_5=3\frac{\sqrt{3}}{8}a^2 \left[q_{0,1}(q_{0,1}-1)+q_{0,2}(1-q_{0,2})+q_{1,1}(1-q_{1,1})-q_{1,2}(1-q_{1,2})\right].
$$
\end{proposition}
\begin{proof}
The proof follows making use of the probability generating function obtained in
Proposition \ref{prop:1}.
\end{proof}
\par
We are now able to state a central limit theorem for the considered random walk.
\begin{proposition}\label{Prop:TCL}
Let $S_n$ be the random vector defined by (\ref{sum}), and $m_n$ its mean given in (\ref{eq:mean}).
Then, as $n\to \infty$,
$$
 n^{-1/2}(S_n- m_n)
$$
converges weakly to the centered bivariate normal distribution with covariance matrix  
\begin{equation}
C
 :=\left(\begin{array}{cc}
 \sigma_1^2&\  \sigma_{1,2}\\[0.1cm]
 \sigma_{1,2}&\  \sigma_2^2
\end{array}\right),
\label{eq:matrxC}
\end{equation}
whose elements are expressed in Proposition \ref{mean_var_cov}.
\end{proposition}
\begin{proof}
The proof follows noting that ${\tilde{X}_i \choose \tilde{Y}_i}$, $i=1,2,\ldots$, are independent and
${\tilde{X}_{2i-1}+\tilde{X}_{2i} \choose \tilde{Y}_{2i-1}+\tilde{Y}_{2i}}$, $i=1,2,\ldots$,
are identically distributed.
\end{proof}
\par
Hereafter we discuss another form of convergence to the bivariate normal distribution, which
involves also a scaling of the hexagonal lattice, and a convergence to Brownian motion.
\begin{remark}\label{rem:differenbt-scaling-limit}
Let $\{ S^*_k(t); \; t\geq 0\}_{k \in \mathbb{N}}$ be a sequence
of continuous-time stochastic processes defined in terms of the
random vector (\ref{sum}), such that
$$
 S^*_k(t)=S_{\left\lfloor t \right\rfloor k}\big|_{a\equiv a/\sqrt{k}}
 =\frac{1}{\sqrt{k}} S_{\left\lfloor t \right\rfloor k}, \qquad t\geq 0.
$$
Hence, from  Proposition \ref{Prop:TCL} it follows that, as $k \to \infty$,
$$
\left\lfloor t \right\rfloor^{-\frac{1}{2}}\left[S^*_k(t)-\mathbb E(S^*_k(t))\right]
$$
converges weakly to  the centered bivariate normal distribution with covariance matrix
(\ref{eq:matrxC}).
\end{remark}
\begin{remark}\label{rem:donsker-theorem-application}
As application of the multidimensional Donsker's Theorem,
and by Proposition \ref{Prop:TCL}, the normalized partial-sum
process
$$
 {\bf S}_n(t):= n^{-1/2}(S_{\lfloor n t  \rfloor} - m_{\lfloor n t\rfloor}), \qquad t\geq 0,
$$
converges weakly to $\textbf{B}D$ (as $n\to \infty$), where
$\textbf{B}$ is the standard $2$-dimensional Brownian motion, and
$D$ is a $2 \times 2$ matrix such that $D^T D= C$, with $C$ defined
by (\ref{eq:matrxC}). In other words we mean
$\textbf{B}D \stackrel{d}{=} \{\textbf{B}(t; 0, C): t\geq 0\}$,
where $\textbf{B}(t; 0, C)$ denotes the $2$-dimensional Brownian
motion with drift vector 0 and covariance matrix $C$. 
\end{remark}
%
\section{Large and Moderate Deviations}
%
We start this section by recalling some well-known basic definitions in large
deviations (see \cite{DemboZeitouni} as a reference on this
topic). A sequence of positive numbers $\{v_n:n\geq 1\}$ is called
speed if $\lim_{n\to\infty}v_n=\infty$. Given a topological space
$\mathcal{Z}$, a lower semi-continuous function
$I:\mathcal{Z}\to[0,\infty]$ is called rate function. 
If the level sets $\{\{z\in\mathcal{Z}:I(z)\leq\eta\}:\eta\geq 0\}$
are compact, the rate function $I$ is said to be good. Finally a
sequence of random variables $\{Z_n:n\geq 1\}$, taking values on a
topological space $\mathcal{Z}$, satisfies the LDP with speed $v_n$ and rate function $I$
if
$$\liminf_{n\to\infty}\frac{1}{v_n}\log P(Z_n\in O)\geq-\inf_{z\in O}I(z)\ \mbox{for all open sets}\ O$$
and
$$\limsup_{n\to\infty}\frac{1}{v_n}\log P(Z_n\in C)\leq-\inf_{z\in O}I(z)\ \mbox{for all closed sets}\ C.$$

In this section we prove two results:
\begin{enumerate}
\item the LDP of
$\left\{\left(\frac{X_n}{n},\frac{Y_n}{n}\right):n\geq 1\right\}$,
with speed $v_n=n$;
\item for all sequences of positive numbers $\{a_n:n\geq 1\}$ such that
\begin{equation}\label{eq:MD-conditions}
a_n\to 0\quad \mbox{and}\quad na_n\to\infty\quad (\mbox{as}\
n\to\infty),
\end{equation}
the LDP of
$\left\{\sqrt{na_n}\left(\frac{X_n-\mathbb{E}(X_n)}{n},\frac{Y_n-\mathbb{E}(Y_n)}{n}\right):n\geq
1\right\}$, with speed function $1/a_n$.
\end{enumerate}
In both cases we prove the LDPs with an application of G\"{a}rtner
Ellis Theorem (see e.g.\ Theorem 2.3.6 in \cite{DemboZeitouni}).

We start with the first result which allows to say that
$$\left(\frac{X_n}{n},\frac{Y_n}{n}\right)\to (\mu_1,\mu_2)\quad (\mbox{as}\ n\to\infty)$$
(note that in particular we have
$$
\frac{\mathbb{E}(X_n)}{n}\to  \mu_1\quad \mbox{and}\quad
\frac{\mathbb{E}(Y_n)}{n}\to  \mu_2\quad (\mbox{as}\ n\to\infty)
$$
by Proposition \ref{mean_var_cov}).

\begin{proposition}\label{prop:LD}
The sequence
$\left\{\left(\frac{X_n}{n},\frac{Y_n}{n}\right):n\geq 1\right\}$
satisfies the LDP, with speed $v_n=n$, and good rate function
$\Lambda^*$ defined by
$$
\Lambda^*(x,y):=\sup_{(\lambda_1,\lambda_2)\in\mathbb{R}^2}\{\lambda_1x+\lambda_2y-\Lambda(\lambda_1,\lambda_2)\},
$$
where
\begin{equation}\label{Lambda}
\Lambda(\lambda_1,\lambda_2)=\frac{1}{2}\log\left(g_0(\lambda_1,\lambda_2)g_1(\lambda_1,\lambda_2)\right),
\end{equation}
and
\begin{eqnarray}\label{g_i}
g_i(\lambda_1,\lambda_2)&=&q_{i,0}+q_{i,1}e^{(-1)^i \sqrt{3}a\left(-\frac{\sqrt{3}}{2}\lambda_1+\frac{1}{2}\lambda_2\right)}\nonumber\\
&&+q_{i,2}e^{(-1)^{i+1}\sqrt{3}a\left(\frac{\sqrt{3}}{2}\lambda_1+\frac{1}{2}\lambda_2\right)},\qquad
i=0,1.
\end{eqnarray}
\end{proposition}
\begin{proof}
We have to check that
$$
\lim_{n\rightarrow+\infty}\frac{1}{n}\log
G(e^{\lambda_1},e^{\lambda_2};n)=\Lambda(\lambda_1,\lambda_2)\
(\mbox{for all}\ (\lambda_1,\lambda_2)\in\mathbb{R}^2),
$$
where $\Lambda$ is the function in (\ref{Lambda}). Then the LDP
will follow from a straightforward application of G\"{a}rtner
Ellis Theorem because the function $\Lambda$ is finite and
differentiable. In order to do that we remark that
$$
G(e^{\lambda_1},e^{\lambda_2};n)=\left\{
\begin{array}{ll}
F_0(e^{\sqrt{3}a\left(\frac{\sqrt{3}}{2}\lambda_1+\frac{1}{2}\lambda_2\right)},e^{\lambda_2
\sqrt{3}a};n),
& \; n \; \textrm{even}\\
&\\
e^{\lambda_1 a}
F_1(e^{\sqrt{3}a\left(\frac{\sqrt{3}}{2}\lambda_1+\frac{1}{2}\lambda_2\right)},e^{\lambda_2
\sqrt{3}a};n), & \; n\; \textrm{odd} 
\end{array}
\right.
$$
by (\ref{prob_gen_func}). Then, by recalling the expression of
$F_i$ in (\ref{F_i_expl}), one has
$$
\frac{1}{n} \log G(e^{\lambda_1},e^{\lambda_2};n)=\left\{
\begin{array}{ll}
\frac{1}{2}\log[g_0(\lambda_1,\lambda_2)g_1(\lambda_1,\lambda_2)],& \textrm{$n$  even}\\
&\\
\frac{1}{n}\lambda_1 a+\frac{1}{2} \log [g_0(\lambda_1,\lambda_2)g_1(\lambda_1,\lambda_2)]&\\
+\frac{1}{n} \log
[g_1^{-1}(\lambda_1,\lambda_2)g_0(\lambda_1,\lambda_2)]^{1/2}, &
\; \textrm{$n$  odd},
\end{array}
\right.
$$
from which the theorem follows by taking the limit for
$n\rightarrow +\infty$.
\end{proof}
%

%
\begin{remark}\label{rem:Lambda-for-a-particular-case}
From (\ref{Lambda}) and (\ref{g_i}), if $q_{i,r}=\frac{1}{3}$ for
all $i=0,1$ and $r=0,1,2$, it results
\begin{eqnarray}
\Lambda(\lambda_1,\lambda_2)&=&
\frac{1}{2}\log \left(\frac{1}{3}+\frac{2}{9}\left\{\cosh\left[\sqrt{3}a\left(\frac{\sqrt{3}}{2}\lambda_1-\frac{1}{2}\lambda_2\right)\right]\right.\right.\nonumber \\
&&\left.\left.+\cosh\left[\sqrt{3}a\left(\frac{\sqrt{3}}{2}\lambda_1+\frac{1}{2}\lambda_2\right)\right]+\cosh (\lambda_2 \sqrt{3}a) \right\}\right).\nonumber
\end{eqnarray}
\end{remark}
\par
Hereafter we remark how the results specified in Proposition \ref{prop:LD}  allow to evaluate some quantities of 
interest in applied contexts, such as estimates of  hitting probabilities that are relevant for simulation purposes.
\begin{remark}\label{rem:Collamore}
	The LDP in Proposition \ref{prop:LD} can be used
		to obtain asymptotic results, as $\varepsilon\to 0$, for some hitting probabilities
		$$\Psi_\varepsilon(A):=P((X_n,Y_n)\in A/\varepsilon\ \mbox{for some}\ n)$$
		(see e.g.\ \cite{Collamore2002}) and for some first passage times
		$$T_\varepsilon(A):=\varepsilon\inf\{n:(X_n,Y_n)\in A/\varepsilon\}$$
		(see e.g.\ \cite{Collamore1998}); here $A$ belongs to a class of subsets of 
		$\mathbb{R}^2$ such that the set $A/\varepsilon$ drifts away to infinity 
		(as $\varepsilon\to 0$), and takes an opposite direction with respect to the vector 
		$(\mu_1,\mu_2)$ (assumed to be different from the null vector). The results in
		\cite{Collamore2002} provide an estimate of $\Psi_\varepsilon(A)$ by Monte Carlo
		simulation using the importance sampling technique.    
\end{remark}
\par
The next proposition  provides a class
of LDPs for centered random variables; specifically we  
consider every sequence of positive numbers
$\{a_n:n\geq 1\}$ such that (\ref{eq:MD-conditions}) holds.
The term used for this class of LDPs is \emph{moderate deviations}.
In some sense we fill the gap between the following two cases:
\begin{itemize}
\item the convergence to $(0,0)$ of 
$\frac{S_n-m_n}{n}=\left(\frac{X_n-\mathbb{E}(X_n)}{n},\frac{Y_n-\mathbb{E}(Y_n)}{n}\right)$;
\item the asymptotic normality result in Proposition \ref{Prop:TCL} above,
(see also Remark \ref{rem:typical-remark-MD} below).
\end{itemize}
Note that, these two cases can be obtained in the following proposition by 
setting $a_n=1/n$ and $a_n=1$, respectively. In both cases only one
condition in (\ref{eq:MD-conditions}) is satisfied. 
\begin{proposition}\label{prop:MD}
For all sequences of positive numbers $\{a_n:n\geq 1\}$ such that
(\ref{eq:MD-conditions}) holds, the sequence
$\left\{\sqrt{na_n}\left(\frac{X_n-\mathbb{E}(X_n)}{n},\frac{Y_n-\mathbb{E}(Y_n)}{n}\right):n\geq
1\right\}$ satisfies the LDP, with speed $1/a_n$, and good rate
function $\tilde{\Lambda}^*$ defined by
\begin{equation}\label{Lambda-tilde-star}
\tilde{\Lambda}^*(x,y):=\sup_{(\lambda_1,\lambda_2)\in\mathbb{R}^2}\{\lambda_1x+\lambda_2y-\tilde{\Lambda}(\lambda_1,\lambda_2)\},
\end{equation}
where
\begin{equation}\label{Lambda-tilde}
\tilde{\Lambda}(\lambda_1,\lambda_2):=\frac{1}{2}(\lambda_1,\lambda_2)C\binom{\lambda_1}{\lambda_2}
\end{equation}
and $C$ is the matrix defined by (\ref{eq:matrxC}), whose elements
are expressed in Proposition \ref{mean_var_cov}.
\end{proposition}
\begin{proof}
We have to check that
$$
\lim_{n\rightarrow+\infty}\tilde{\Lambda}_n(\lambda_1,\lambda_2)=\tilde{\Lambda}(\lambda_1,\lambda_2)\
(\mbox{for all}\ (\lambda_1,\lambda_2)\in\mathbb{R}^2),
$$
where
$$\tilde{\Lambda}_n(\lambda_1,\lambda_2):=a_n\left(\log
G(e^{\lambda_1/\sqrt{na_n}},e^{\lambda_2/\sqrt{na_n}};n)-\frac{\lambda_1\mathbb{E}(X_n)+\lambda_2\mathbb{E}(Y_n)}{\sqrt{na_n}}\right)$$
and $\tilde{\Lambda}$ is the function in (\ref{Lambda-tilde}).
Then the LDP will follow from a straightforward application of
G\"{a}rtner Ellis Theorem because the function $\Lambda$ is finite
and differentiable. In order to do that, by taking into account
$na_n\to\infty$ (by (\ref{eq:MD-conditions})), we consider the Mac
Laurin formula of order 2 of
$$
\Psi_n(\lambda_1,\lambda_2):=\log
G\left(e^{\lambda_1/\sqrt{na_n}},e^{\lambda_2/\sqrt{na_n}};n\right)
$$
namely
\begin{multline*}
\Psi_n(\lambda_1,\lambda_2)=\frac{\lambda_1\mathbb{E}(X_n)+\lambda_2\mathbb{E}(Y_n)}{\sqrt{na_n}} \\
+\frac{1}{2}\left(\frac{\lambda_1^2}{na_n}Var(X_n)+\frac{\lambda_2^2}{na_n}Var(Y_n)+2\frac{\lambda_1\lambda_2}{na_n}Cov(X_n,Y_n)
+o\left(\frac{1}{na_n}\right)\right).
\end{multline*}
Then we obtain
\begin{multline*}
\tilde{\Lambda}_n(\lambda_1,\lambda_2)=a_n\left(\frac{1}{2}\left(\frac{\lambda_1^2}{na_n}Var(X_n)+\frac{\lambda_2^2}{na_n}Var(Y_n)\right.\right.\\
\left.\left.+2\frac{\lambda_1\lambda_2}{na_n}Cov(X_n,Y_n)+o\left(\frac{1}{na_n}\right)\right)\right)\\
=\frac{1}{2}\left(\frac{\lambda_1^2}{n}Var(X_n)+\frac{\lambda_2^2}{n}Var(Y_n)+2\frac{\lambda_1\lambda_2}{n}Cov(X_n,Y_n)+o\left(\frac{1}{n}\right)\right).
\end{multline*}
We conclude noting that the desired limit holds because
$$
 \frac{Var(X_n)}{n}\to  \sigma_1^2,\ \frac{Var(Y_n)}{n}\to  \sigma_2^2,\ \frac{Cov(X_n,Y_n)}{n}\to  \sigma_{1,2}\ (\mbox{as}\ n\to\infty)
$$
by Proposition \ref{mean_var_cov}, and by taking into account the
function $\tilde{\Lambda}$ in (\ref{Lambda-tilde}).
\end{proof}
\begin{remark}\label{rem:invertible-C-for-MD}
If $C$ is invertible, then the supremum in
(\ref{Lambda-tilde-star}) is attained at
$\binom{\lambda_1}{\lambda_2}=C^{-1}\binom{x}{y}$ and we get
$$
 \tilde{\Lambda}^*(x,y)=\frac{1}{2}(x,y)C^{-1}\binom{x}{y}\ (\mbox{for all}\ (x,y)\in\mathbb{R}^2).
$$
On the other hand, if $C$ is not invertible, we can
consider the restriction $\hat{C}$, say, of $C$ to
its range $\mathrm{Im}(C)$; then it is invertible with
inverse $\hat{C}^{-1}$ and thus we have
$$
\tilde{\Lambda}^*(x,y)=\left\{
\begin{array}{ll}
\frac{1}{2}(x,y)\hat{C}^{-1}\binom{x}{y},&\ \mbox{if}\ 
(x,y)\in\mathrm{Im}(C)\\
\infty,&\ \mbox{otherwise}.
\end{array}\right.
$$
\end{remark}
\begin{remark}\label{rem:typical-remark-MD}
A close inspection of the proof of Proposition \ref{prop:MD} reveals that 
all the computations for moderate deviations  work well even if $a_n=1$; note that in 
such a case the first condition in (\ref{eq:MD-conditions}) fails. Thus, for all
$(\lambda_1,\lambda_2)\in\mathbb{R}^2$,
$$
 \lim_{n\to\infty}\log \mathbb{E}\left[\exp\left(\lambda_1\frac{X_n-\mathbb{E}(X_n)}{\sqrt{n}}+\lambda_2\frac{Y_n-\mathbb{E}(Y_n)}{\sqrt{n}}\right)\right]
 =\frac{1}{2}(\lambda_1,\lambda_2)C\binom{\lambda_1}{\lambda_2},
$$
and therefore 
$\left\{\left(\frac{X_n-\mathbb{E}(X_n)}{\sqrt{n}},\frac{Y_n-\mathbb{E}(Y_n)}{\sqrt{n}}\right):n\geq 1\right\}$
converges weakly (as $n\to\infty$) to the centered bivariate Normal distribution 
with covariance matrix $C$ (cfr.\ Proposition \ref{Prop:TCL}).
\end{remark}

%
\section{A First-Passage-Time Problem}
%
In this section we study the first-passage-time problem of the random walk $(X_n, Y_n)$ 
through suitable boundaries, under specific symmetry conditions. 
\par
In order to deal with more general instances, hereafter we assume that the initial state is arbitrary, and 
belonging to a vertex of $\mathscr{V}_0$. Thus, the initial condition is expressed as 
\begin{equation}
 \mathbb{P}\Big[(X_0,Y_0)
 =\Big(\frac{3}{2}a \tilde j, \frac{\sqrt{3}}{2}a \tilde j+\sqrt{3} a \tilde k\Big)\Big]=1,\qquad 
 \tilde j, \tilde k \in \mathbb{Z}. 
 \label{eq:newinitstate}
\end{equation}
In this section, 
recalling the initial condition (\ref{eq:newinitstate}), we denote by $\mathbb{P}_0(A)$ the probability of 
$A$ conditional on $(X_0,Y_0) =\Big(\frac{3}{2}a \tilde j, \frac{\sqrt{3}}{2}a \tilde j+\sqrt{3} a \tilde k\Big)$. 
Moreover, we adopt the following notation for the state probabilities at time $n\in\mathbb{N}_0$:
\begin{equation}
 p_{j,k}(n|\tilde j, \tilde k):=\mathbb{P}_0\Big[(X_{n},Y_{n})=\Big(\frac{3}{2} a j+i_n a,
 \frac{\sqrt{3}}{2} a j +\sqrt{3}a k\Big)\Big], 
 \label{eq:newprobstate}
\end{equation}
for $\tilde j, \tilde k \in \mathbb{Z}$, where $i_n$ is defined by (\ref{eq:i_n}). 
Clearly, due to the spatial homogeneity of random walk, the  state probabilities (\ref{p_j_k}) and 
(\ref{eq:newprobstate}) are related by the following equality 
\begin{equation}\label{simmp_j_k}
 p_{j-\tilde j,k-\tilde k}(n)= p_{j,k}(n|\tilde j, \tilde k),
 \qquad j,k,\tilde j,\tilde k\in \mathbb{Z},\quad n\in \mathbb{N}_0.
\end{equation}
\par
Let us now introduce the straight-line boundary 
\begin{equation}\label{eq:boundary}
 {\cal B}_r=\Big\{\Big(\frac{3}{2}a r , \frac{\sqrt{3}}{2}a r+\sqrt{3} a k\Big); k\in \mathbb{Z}\Big\},
 \qquad r\in \mathbb Z. 
\end{equation}
Clearly, ${\cal B}_r$ is constituted by vertices of $\mathscr{V}_0$. 
For a fixed $r\in\mathbb{Z}$, and $\tilde j,\tilde k\in \mathbb{Z}$ let
\begin{equation}\label{eq:Trjk}
T_{r}(\tilde j,\tilde k)=\min\{n\in \mathbb{N} :  (X_n, Y_n)\in {\cal B}_r  \}, 
 \qquad (X_0, Y_0)=\Big(\frac{3}{2}a \tilde j, \frac{\sqrt{3}}{2}a \tilde j+\sqrt{3} a \tilde k\Big)
 \not \in {\cal B}_r
\end{equation}
be the first-passage time of $(X_n, Y_n)$ through the boundary (\ref{eq:boundary}). 
Since the initial state belongs to $\mathscr{V}_0$, then the first passage through ${\cal B}_r$ 
may occur only at even times. 
Let $r,s,\tilde j,\tilde k\in \mathbb{Z}$, and $n\in \mathbb{N}$;
we now introduce the first-passage-time probabilities 
\begin{equation}
 g(r,s; n|\tilde j,\tilde k)
 =\mathbb{P}_0\Big[ T_{r}(\tilde j,\tilde k)=n, 
 (X_{n},Y_{n})=\Big(\frac{3}{2} a r ,
 \frac{\sqrt{3}}{2} a r +\sqrt{3}a s\Big)\Big], 
 \label{eq:probg}
\end{equation}
and 
\begin{equation}
 h(r; n|\tilde j,\tilde k)
 =\mathbb{P}_0\Big[ T_{r}(\tilde j,\tilde k)=n \Big].
\label{eq:probh}
\end{equation}
Hence, from (\ref{eq:probg}) and (\ref{eq:probh}) we have 
\begin{equation}
 h(r; n|\tilde j,\tilde k)
 =\sum_{s\in \mathbb Z}
 g(r,s; n|\tilde j,\tilde k).
\label{eq:relprobhg}
\end{equation}
Recalling that the first passage through ${\cal B}_r$ eventually occurs at even times, it follows
$$
  g(r,s; 2n-1|\tilde j,\tilde k)= 0 
  \qquad \hbox{and} 
  \qquad h(r; 2n-1|\tilde j,\tilde k)=0, \qquad n\in \mathbb{N}.
$$ 
Morever, the nature of the sample-paths and the Markov property 
of $(X_n, Y_n)$ allow to write the following relation, for $n\in \mathbb N$, 
\begin{equation}
 p_{j,k}(n\,|\,\tilde j, \tilde k)
 =\sum_{i=1}^n \sum_{s\in \mathbb Z} g(r,s; i \,|\,\tilde j,\tilde k) p_{j,k}(n-i \,|\, r,s)
 \qquad (j\leq r, \tilde j>r) \ \ \hbox{or} \ \ (j\geq r, \tilde j<r).
\label{eq:relprobhgnew}
\end{equation}
%
%
We remark that Eq.\  (\ref{eq:relprobhgnew}) expresses that all sample paths that start on the right (left) of the 
boundary ${\cal B}_r$ and terminate at time $n$ on the  left (right) are forced to cross ${\cal B}_r$ in some point, say $(r,s)$, 
at some instant, say $i$, with $1\leq i\leq n$. Hence, in this case the first passage through ${\cal B}_r$ is certain.  
\par
As for illustration of Eq.\ (\ref{eq:relprobhgnew}), Figure \ref{fig:samplepath} shows the projection 
on the hexagonal lattice of a sample path that crosses the boundary ${\cal B}_r$. 
\begin{figure}[ht] 
\begin{center}
\begin{picture}(341,221)
\put(50,15){\circle{5}}
\put(130,15){\circle{5}}
\put(210,15){\circle{5}}
\put(290,15){\circle{5}}
\put(75,15){\circle*{5}}
\put(155,15){\circle*{5}}
\put(235,15){\circle*{5}}
\put(315,15){\circle*{5}}
\put(50,60){\circle{5}}
\put(130,60){\circle{5}}
\put(210,60){\circle{5}}
\put(290,60){\circle{5}}
\put(75,60){\circle*{5}}
\put(155,60){\circle*{5}}
\put(235,60){\circle*{5}}
\put(315,60){\circle*{5}}
\put(50,105){\circle{5}}
\put(130,105){\circle{5}}
\put(210,105){\circle{5}}
\put(290,105){\circle{5}}
\put(75,105){\circle*{5}}
\put(155,105){\circle*{5}}
\put(235,105){\circle*{5}}
\put(315,105){\circle*{5}}
\put(50,150){\circle{5}}
\put(130,150){\circle{5}}
\put(210,150){\circle{5}}
\put(290,150){\circle{5}}
\put(75,150){\circle*{5}}
\put(155,150){\circle*{5}}
\put(235,150){\circle*{5}}
\put(315,150){\circle*{5}}
\put(90,37.5){\circle{5}}
\put(170,37.5){\circle{5}}
\put(250,37.5){\circle{5}}
\put(330,37.5){\circle{5}}
\put(35,37.5){\circle*{5}}
\put(115,37.5){\circle*{5}}
\put(195,37.5){\circle*{5}}
\put(275,37.5){\circle*{5}}
\put(90,82.5){\circle{5}}
\put(170,82.5){\circle{5}}
\put(250,82.5){\circle{5}}
\put(330,82.5){\circle{5}}
\put(35,82.5){\circle*{5}}
\put(115,82.5){\circle*{5}}
\put(195,82.5){\circle*{5}}
\put(275,82.5){\circle*{5}}
\put(90,127.5){\circle{5}}
\put(170,127.5){\circle{5}}
\put(250,127.5){\circle{5}}
\put(330,127.5){\circle{5}}
\put(35,127.5){\circle*{5}}
\put(115,127.5){\circle*{5}}
\put(195,127.5){\circle*{5}}
\put(275,127.5){\circle*{5}}

\put(90,172.5){\circle{5}}
\put(170,172.5){\circle{5}}
\put(250,172.5){\circle{5}}
\put(115,172.5){\circle*{5}}
\put(195,172.5){\circle*{5}}
\put(275,172.5){\circle*{5}}
%
\put(250,-2.5){\line(0,1){5}}
\put(250,7.5){\line(0,1){5}}
\put(250,17.5){\line(0,1){5}}
\put(250,27.5){\line(0,1){5}}
\put(250,42.5){\line(0,1){5}}
\put(250,52.5){\line(0,1){5}}
\put(250,62.5){\line(0,1){5}}
\put(250,72.5){\line(0,1){5}}
\put(250,87.5){\line(0,1){5}}
\put(250,97.5){\line(0,1){5}}
\put(250,107.5){\line(0,1){5}}
\put(250,117.5){\line(0,1){5}}
\put(250,132.5){\line(0,1){5}}
\put(250,142.5){\line(0,1){5}}
\put(250,152.5){\line(0,1){5}}
\put(250,162.5){\line(0,1){5}}
\put(250,177.5){\line(0,1){5}}
\put(250,187.5){\line(0,1){5}}
%
\gra{
\put(50,60){\line(1,0){25}}
\put(50,60){\line(-2,3){15}}
\put(50,60){\line(-2,-3){15}}
\put(130,60){\line(1,0){25}}
\put(130,60){\line(-2,3){15}}
\put(130,60){\line(-2,-3){15}}
\put(210,60){\line(1,0){25}}
\put(210,60){\line(-2,3){15}}
\put(210,60){\line(-2,-3){15}}
\put(50, 105){\line(1,0){25}}
\put(50, 105){\line(-2,3){15}}
\put(50, 105){\line(-2,-3){15}}
\put(130, 105){\line(1,0){25}}
\put(130, 105){\line(-2,3){15}}
\put(130, 105){\line(-2,-3){15}}
\put(210, 105){\line(1,0){25}}
\put(210, 105){\line(-2,3){15}}
\put(210, 105){\line(-2,-3){15}}
\put(90, 82.5){\line(1,0){25}}
\put(90, 82.5){\line(-2,3){15}}
\put(90, 82.5){\line(-2,-3){15}}
\put(170, 82.5){\line(1,0){25}}
\put(170, 82.5){\line(-2,3){15}}
\put(170, 82.5){\line(-2,-3){15}}
\put(250, 82.5){\line(1,0){25}}
\put(250, 82.5){\line(-2,3){15}}
\put(250, 82.5){\line(-2,-3){15}}
\put(130, 150){\line(1,0){25}}
\put(130, 150){\line(-2,3){15}}
\put(130, 150){\line(-2,-3){15}}
\put(210, 150){\line(1,0){25}}
\put(210, 150){\line(-2,3){15}}
\put(210, 150){\line(-2,-3){15}}
\put(290, 150){\line(-2,3){15}}
\put(290, 150){\line(-2,-3){15}}
\put(90, 127.5){\line(1,0){25}}
\put(90, 127.5){\line(-2,3){15}}
\put(90, 127.5){\line(-2,-3){15}}
\put(170, 127.5){\line(1,0){25}}
\put(170, 127.5){\line(-2,3){15}}
\put(170, 127.5){\line(-2,-3){15}}
\put(250, 127.5){\line(1,0){25}}
\put(250, 127.5){\line(-2,3){15}}
\put(250, 127.5){\line(-2,-3){15}}
\put(35, 127.5){\line(2,3){15}}
\put(50, 150){\line(1,0){25}}
\put(75, 150){\line(2,3){15}}
\put(90, 172.5){\line(1,0){25}}
\put(155, 150){\line(2,3){15}}
\put(170, 172.5){\line(1,0){25}}
\put(235, 150){\line(2,3){15}}
\put(250, 172.5){\line(1,0){25}}
\put(275, 127.5){\line(2,-3){15}}
\put(275, 82.5){\line(2,3){15}}
\put(35, 37.5){\line(2,-3){15}}
\put(50, 15){\line(1,0){25}}
\put(75, 15){\line(2,3){15}}
\put(75, 15){\line(2,3){15}}
\put(90, 37.5){\line(-2,3){15}}
\put(90, 37.5){\line(1,0){25}}
\put(115, 37.5){\line(2,-3){15}}
\put(130, 15){\line(1,0){25}}
\put(155, 15){\line(2,3){15}}
\put(170, 37.5){\line(-2,3){15}}
\put(170, 37.5){\line(1,0){25}}
\put(195, 37.5){\line(2,-3){15}}
\put(210, 15){\line(1,0){25}}
\put(235, 15){\line(2,3){15}}
\put(250, 37.5){\line(-2,3){15}}
\put(250, 37.5){\line(1,0){25}}
\put(275, 37.5){\line(2,3){15}}
\put(275, 37.5){\line(2,3){15}}
\put(290, 60){\line(-2,3){15}}
\put(290, 15){\line(1,0){25}}
\put(290, 60){\line(1,0){25}}
\put(290, 105){\line(1,0){25}}
\put(290, 150){\line(1,0){25}}
\put(290, 15){\line(-2,3){15}}
\put(315, 15){\line(2,3){15}}
\put(315, 60){\line(2,3){15}}
\put(315, 105){\line(2,3){15}}
\put(315, 60){\line(2,-3){15}}
\put(315, 105){\line(2,-3){15}}
\put(315, 150){\line(2,-3){15}}
}
\put(130, 60){\line(1,0){25}}
\put(155, 60){\line(2,3){15}}
\put(170, 82.5){\line(1,0){25}}
\put(195, 82.5){\line(2,3){15}}
\put(210, 105){\line(1,0){25}}
\put(235, 105){\line(2,-3){15}}
\put(250, 82.5){\line(1,0){25}}
\put(275, 82.5){\line(2,-3){15}}
%
%
\put(225,195){\makebox(50,15)[t]{${\cal B}_r$}}
\put(110,60){\makebox(50,15)[t]{{\small $(\tilde j, \tilde k)$}}}
\put(270,60){\makebox(50,15)[t]{{\small $( j, k)$}}}
\put(234,80){\makebox(50,15)[t]{{\small $( r, s)$}}}
\end{picture}
\end{center}
\caption{ 
Projection on the hexagonal lattice of a sample path of $(X_n, Y_n)$ that crosses the boundary ${\cal B}_r$, 
with $j\geq r, \tilde j<r$. 
}
\label{fig:samplepath}
\end{figure}
\par
For any fixed $r\in\mathbb{Z}$, and for $j, k, \tilde j, \tilde k \in \mathbb{Z}$ let us now introduce the following ``taboo probability'':
\begin{equation}
 p^{\langle r\rangle}_{j,k}(n\,|\,\tilde j, \tilde k)
 =\mathbb{P}_0\Big[ (X_{n},Y_{n})=\Big(\frac{3}{2} a j+i_n a,
 \frac{\sqrt{3}}{2} a j +\sqrt{3}a k\Big), T_{r}(\tilde j,\tilde k)>n \Big], 
 \label{equation:24}
\end{equation}
with $\Big(\frac{3}{2}a \tilde j, \frac{\sqrt{3}}{2}a \tilde j+\sqrt{3} a \tilde k\Big)\not \in {\cal B}_r$,
for  $i_n$  defined by (\ref{eq:i_n}). Note that such taboo probability refers to a transition that ends   
at time $n$ and that avoids the boundary (\ref{eq:boundary}), conditional on the initial state belonging to $\mathscr{V}_0$. 
\par
We are now able to express the taboo probability in terms of transition probabilities, under specific assumptions on the 
one-step transition probabilities.  
\begin{theorem}\label{th:probtaboo}
Let $r\in \mathbb Z$. 
If $q_{0,1}=q_{1,2}$ and $q_{1,1}=q_{0,2}$, then for $n\in \mathbb N$ we have 
\begin{equation}
 p^{\langle r\rangle}_{j,k}(2n\,|\,\tilde j, \tilde k)
 = p_{j,k}(2n\,|\,\tilde j, \tilde k) 
 - p_{2r-j, j-r+k}(2n\,|\,\tilde j, \tilde k), 
\label{equation:11}
\end{equation}
with $j<r$, $\tilde j<r$ or  $j>r$, $\tilde j>r$.
\label{teorema1}
\end{theorem}
\begin{proof}
Noting that $g(r,s; 2 i +1|\tilde j,\tilde k)=0$ for $i\in \mathbb N_0$, 
for $k, \tilde k\in \mathbb Z$, 
the following relation holds true for  $j<r$, $\tilde j<r$ or  $j>r$, $\tilde j>r$: 
\begin{equation}
 p^{\langle r\rangle}_{j,k}(2n\,|\,\tilde j, \tilde k)
 =p_{j,k}(2n\,|\,\tilde j, \tilde k)
 -\sum_{i=1}^{\lfloor  n/2\rfloor} \sum_{s\in\mathbb{Z}} 
  g(r,s; 2 i \,|\,\tilde j,\tilde k) p_{j-r,k-s}(2n-2i \,|\, 0,0). 
\label{equation:11bis}
\end{equation}
Indeed, the sample-paths that avoid the boundary (\ref{eq:boundary}) are obtained by considering all sample-paths 
and excluding those that cross ${\cal B}_r$. 
We take into account that, under the given assumptions, the following relation holds due to case (i)
of Corollary \ref{coroll:symm}:
$$p_{j-r,k-s}(2n-2i | 0,0)=p_{2r-j,j-r+k}(2n-2i | r,s).$$ 
Hence, recalling  (\ref{eq:relprobhgnew}), the last term of (\ref{equation:11bis}) can be recognized to be a 
transition probability. The equation (\ref{equation:11}) thus finally follows. 
\end{proof}
\par
Let us now obtain a closed-form expression of the first-passage-time probabilities (\ref{eq:probg}) 
under the symmetry conditions $q_{0,1}=q_{1,2}$ and $q_{1,1}=q_{0,2}$. 
\begin{theorem}\label{teorema2}
Let $n\in \mathbb{N}_0$ and $s, \tilde k \in \mathbb{Z}$. 
Under the assumption of Theorem \ref{th:probtaboo}, for  $\tilde j<r$ we have 
\begin{eqnarray}
  g(r,s; 2n+2|\tilde j,\tilde k) \!\!\!\!
  & = & \!\!\!\! q_{0,0} \left\{q_{1,2} \left[ p_{r-1,s}(2n\,|\,\tilde j, \tilde k) - p_{r+1,s-1}(2n\,|\,\tilde j, \tilde k)  \right] \right.
  \nonumber \\
 & + & \!\!\!\! \left. q_{1,1}  \left[ p_{r-1,s+1}(2n\,|\,\tilde j, \tilde k) - p_{r+1,s}(2n\,|\,\tilde j, \tilde k)  \right]\right\}.
 \label{equation:13}
\end{eqnarray}
For    $\tilde j>r$ we have 
\begin{eqnarray*}
  g(r,s; 2n+2|\tilde j,\tilde k) \!\!\!\!
  & = & \!\!\!\! q_{1,0} \left\{q_{0,2} \left[ p_{r+1,s}(2n\,|\,\tilde j, \tilde k) - p_{r-1,s+1}(2n\,|\,\tilde j, \tilde k)  \right] \right.
  \nonumber \\
 & + & \!\!\!\! \left. q_{0,1}  \left[ p_{r+1,s-1}(2n\,|\,\tilde j, \tilde k) - p_{r-1,s}(2n\,|\,\tilde j, \tilde k)  \right]\right\}.
 \label{equation:13new}
\end{eqnarray*}
\end{theorem}
\begin{proof}
Recalling (\ref{eq:probg}), for   $\tilde j<r$ and 
$s, \tilde k\in \mathbb Z$, the first-passage-time probabilities 
can be expressed in terms of the taboo probability (\ref{equation:24}) as 
$$
 g(r,s; 2n+2|\tilde j,\tilde k)
 =q_{0,0} \left[
 q_{1,2} \,p^{\langle r\rangle}_{r-1,s}(2n\,|\,\tilde j, \tilde k)
 +q_{1,1} \,p^{\langle r\rangle}_{r-1,s+1}(2n\,|\,\tilde j, \tilde k)\right], 
 \qquad n\in \mathbb{N}_0.
$$ 
Making use of Theorem \ref{teorema1}  we thus obtain (\ref{equation:13}). 
The case   $\tilde j>r$ can be proved similarly. 
\end{proof}
\par
Figures \ref{fig:fptg4}, \ref{fig:fptg5} and \ref{fig:fptg6} show some plots of the first-passage-time probabilities obtained 
in Theorem \ref{teorema2}. The shapes of the probability distributions reflect the position of the initial state 
with respect to the straight-line boundary. 
Moreover, we note that the choices of $q_{i,j}$ in Figures \ref{fig:fptg4}$\div$\ref{fig:fptg6} 
implies that $\mu_1=0$ and $\mu_2=a\sqrt{3}/12$, due to Proposition \ref{mean_var_cov}. 
Hence, the random walk is biased toward large values of the $y$-coordinate, and thus 
the probabilities $g(r,s; 2n+2|0,0)$ shift to the right  with increasing $r$.  
%
\begin{figure}[t]
\centering
\includegraphics[width=0.95\textwidth]{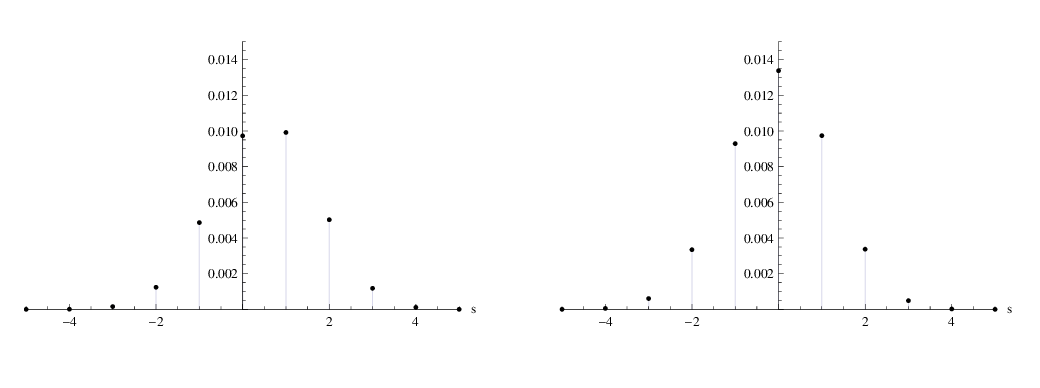}\\
\caption{Plot of $g(r,s; 12|0,0)$ 
with $q_{0,0}=q_{1,0}=1/2$, $q_{0,1}=q_{1,2}=1/3$ 
and $q_{0,2}=q_{1,1}=1/6$ 
for $r=1$ (on the left) and $r=2$ (on the right).}
\label{fig:fptg4}
\end{figure}
%
\begin{figure}[t]
\centering
\includegraphics[width=0.95\textwidth]{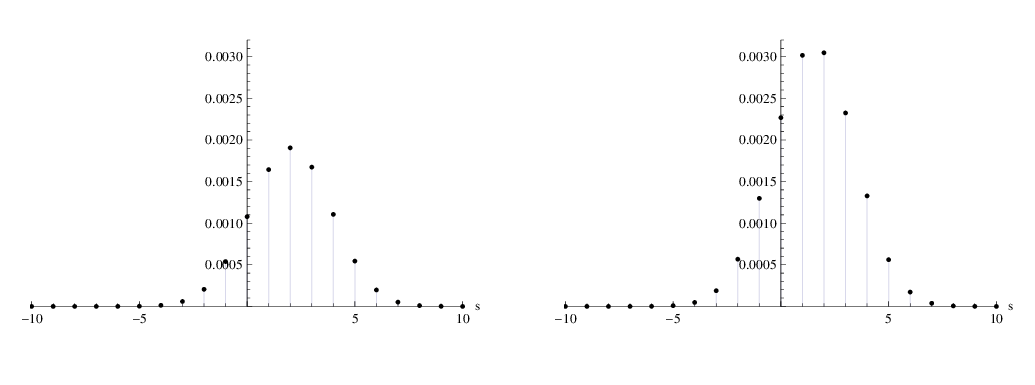}\\
\caption{Plot of $g(r,s; 30|0,0)$ 
with $q_{0,0}=q_{1,0}=1/2$, $q_{0,1}=q_{1,2}=1/3$ 
and $q_{0,2}=q_{1,1}=1/6$ 
for $r=1$ (on the left) and $r=2$ (on the right).}
\label{fig:fptg5}
\end{figure}
%
\begin{figure}[t]
\centering
\includegraphics[width=0.95\textwidth]{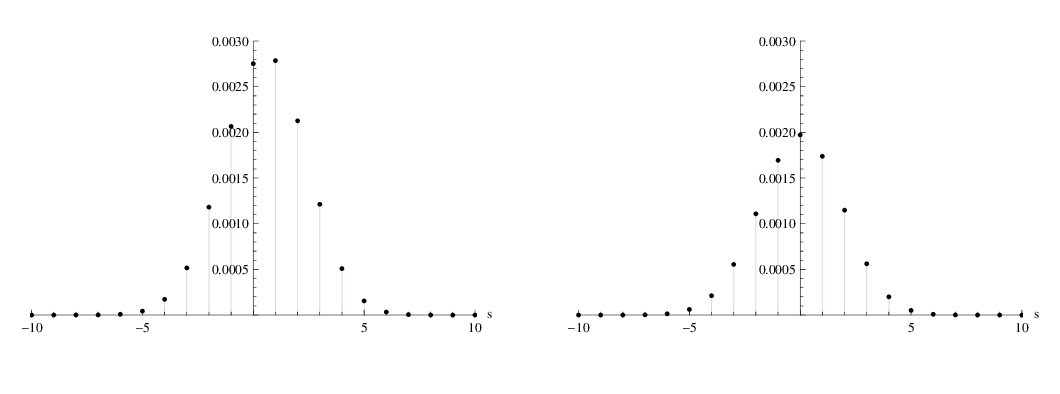}\\
\caption{Plot of $g(r,s; 30|0,0)$ 
with $q_{0,0}=q_{1,0}=1/2$, $q_{0,1}=q_{1,2}=1/3$ 
and $q_{0,2}=q_{1,1}=1/6$ 
for $r=4$ (on the left) and $r=5$ (on the right).}
\label{fig:fptg6}
\end{figure}
\par
The results obtained in Theorems \ref{th:probtaboo} and \ref{teorema2}, which rely on the 
symmetry property given in point (i) of Corollary \ref{coroll:symm}, can be extended to other forms of boundaries 
by exploiting the symmetries in points (ii) and (iii) of such corollary. 
\par
With reference to (\ref{eq:probh}) let us now introduce the following cumulative distribution function, for 
$r,\tilde j,\tilde k\in \mathbb{Z}$ and $\tau\in \mathbb{N}_0$, 
\begin{equation}
 H(r; 2\tau+2|\tilde j,\tilde k)
 =\mathbb{P}_0\Big[ T_{r}(\tilde j,\tilde k)\leq 2\tau+2 \Big]
 =\sum_{n=0}^{\tau} h(r; 2n+2|\tilde j,\tilde k).
 \label{eq:probcdfH}
\end{equation}
As example, we show  some plots of function (\ref{eq:probcdfH}) in Figure \ref{fig:fptG}, 
for different choices of $r$. Clearly, it shows that $H(r; 2\tau+2|\tilde j,\tilde k)$ is 
decreasing when $r$ increases, and thus the boundary ${\cal B}_r$ runs far from the initial state. 
%
\begin{figure}[t]
\centering
\includegraphics[width=0.5\textwidth]{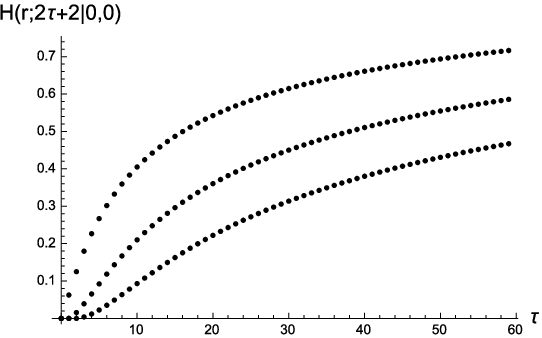}\\
\caption{
Plot of $H(r; 2\tau+2|0,0)$ 
with $q_{0,0}=q_{1,0}=1/2$, $q_{0,1}=q_{1,2}=1/3$ 
and $q_{0,2}=q_{1,1}=1/6$ 
for $r=2$, $3$, $4$  (from top to bottom). 
}
\label{fig:fptG}
\end{figure}
\par
In the chemistry literature wide interest is given to the self-avoiding random walks on lattices, 
i.e.\ walks subject to the condition that no lattice site may be visited more than once (see, for instance, 
\cite{Domb}). Indeed, in polymer science self-avoiding random walks are used as a simple way to describe 
polymeric chains. Unfortunately, mathematical calculations concerning such walks are quite hard, 
since the sample paths do not intersect themselves. In this setting, in analogy with (\ref{eq:Trjk}), 
let us denote by $T_{r}^{\rm sa}(\tilde j,\tilde k)$ the first-passage time through ${\cal B}_r$ when 
the random walk $(X_n, Y_n)$ does not visit the hexagonal lattice sites more than once. 
Hence, since the number of self-avoiding random walks is less than the regular one, 
it is obvious that 
$$
 \mathbb{P}_0[ T_{r}^{\rm sa}(\tilde j,\tilde k)\leq 2\tau+2] \geq H(r; 2\tau+2|\tilde j,\tilde k)
 \qquad \hbox{for all $r,\tilde j,\tilde k\in \mathbb{Z}$ and $\tau\in \mathbb{N}_0$.}
$$ 
In other terms, we have $T_{r}^{\rm sa}(\tilde j,\tilde k) \leq_{\rm st} T_{r} (\tilde j,\tilde k)$,
where $ \leq_{\rm st} $ denotes the usual stochastic order (see, for instance, \cite{Shaked}). 
Consequently, making use of Theorem \ref{teorema2} and Eqs.\ (\ref{eq:relprobhg}) and (\ref{eq:probcdfH}),
under the assumptions of Theorem \ref{th:probtaboo} we can obtain useful lower bounds to 
$\mathbb{P}_0[ T_{r}^{\rm sa}(\tilde j,\tilde k)\leq 2\tau+2]$. Such information is useful to 
investigate the first-reaching-time of a preassigned length along the $x$-axis for two-dimensional 
polymeric chains that grow according to random rules as those of the random walk on hexagonal lattices. 
\par
For the cases treated in Figure \ref{fig:fptG}, from (\ref{eq:probcdfH}) one has 
$H(2; 122 | 0,0)=0.7185$, 
$H(3; 122 | 0,0)=0.5887$, 
$H(4; 122 | 0,0)=0.4709$, for instance. These values provide suitable lower bounds to the probability that 
a polymeric chain described by a self-avoiding random walk, starting at the origin of the hexagonal lattice and 
growing according to the probabilistic rules given in Figure \ref{fig:fptG}, before 122 steps 
reaches a location situated along the line ${\cal B}_2$, ${\cal B}_3$, ${\cal B}_4$, respectively. 

\section{Concluding Remarks}
Random walks on hexagonal lattices have a long history. The properties usually studied are concerning the 
transient and the recurrent behavior. In this paper, by exploiting the fact that the hexagonal lattice is a bipartite graph, 
we have developed a generating function-based approach to obtain closed-form expression for the state probabilities 
of the random walk. The asymptotic behavior of such stochastic process has also been treated in terms of 
large deviation and moderate deviation theory. Finally, since the state probabilities satisfy customary symmetry 
properties (see Corollary \ref{coroll:symm}) we have been able to obtain suitable first-passage-time probabilities 
in a closed form.
\par
Specifically, the results obtained in Section 5 can be applied to several contexts. Indeed, first-passage-time problems 
deserve interest in a variety of situations where the first reaching time of critical states has a relevant role.
For instance, with reference to the models treated by \cite{Prasad} and \cite{Pulliam}, 
first-passage times of random walks on a hexagonal lattice can describe the exit times from assigned regions 
of animals searching for resources. Similarly, in the honeycomb Poisson-Voronoi access cellular network model, 
first-passage times are useful to model the reaching of areas with weak signal for mobile users. 
Another example arising in polymer science has been mentioned at the end of Section 5. 
\par
Possible future developments of the present investigation deal with developing asymptotic results for hitting 
probabilities through regions of different form, along the line mentioned in Remark \ref{rem:Collamore}. 

\section*{Acknowledgments}

This work was partially supported by the groups GNAMPA and GNCS of INdAM (Istituto Nazionale di Alta Matematica), 
by  MIUR - PRIN 2017, project ``Stochastic Models
for Complex Systems'', no.\ 2017JFFHSH, 
by the MIUR Excellence Department Project awarded to
the Department of Mathematics, University of Rome Tor
Vergata (CUP E83C18000100006), and by University of Rome
Tor Vergata (research programme "Mission: Sustainability",
project ISIDE (grant no. E81I18000110005)). 

\par
The authors warmly thanks two anonymous referees and the associate editor for their useful comments that 
improved the paper.

\end{document}